\documentclass[amstex,10pt,reqno]{amsart}
\usepackage{amsmath,amsfonts,amssymb,amsthm,enumerate,hyperref,multicol, graphicx, wrapfig, pdfsync,array, caption, subcaption, color}
\usepackage[dvipsnames]{xcolor}
\newtheorem{lemma}{Lemma}
\newtheorem{thm}{Theorem}

\newtheorem{corollary}{Corollary}

\numberwithin{equation}{section}
\begin{document}
	\leftline{ \scriptsize \it}
	\title[]
	{Improved Convergence and Approximation properties of Baskakov-Durrmeyer Operators}
	\maketitle\
	{\bf \begin{center}
			{\bf  Jaspreet Kaur$^{1,2,a}$ and Meenu Goyal$^{2,b}$}
			\vskip 0.2in
				$^{1}$Department of Mathematics\\
			GLA University\\
			 Mathura-281406, India\\
			\vskip0.2in
			$^{2}$Department of Mathematics\\
			Thapar Institute of Engineering and Technology\\
			Patiala-147004, India\\
			\vskip0.2in
			${}^a$jazzbagri3@gmail.com and ${}^b$meenu\_rani@thapar.edu
	\end{center}}
	\begin{abstract}
	In this paper, we describe two novel changes to the Baskakov-Durrmeyer operators that improve their approximation performance. These improvements are especially designed to produce higher rates of convergence, with orders of one or two. This is a major improvement above the linear rate of convergence commonly associated with conventional Baskakov-Durrmeyer operators. Our research goes thoroughly into the approximation features of these modified operators, providing a thorough examination of their convergence behavior. We concentrate on calculating precise convergence rates, providing thorough error estimates that demonstrate the new operators' efficiency as compared to the classical version. In addition, we construct Voronovskaja-type formulae for these operators, which provide insights into the asymptotic behavior of the approximation process as the operator's degree grows. By exploring these aspects, we demonstrate that the proposed modifications not only surpass the classical operators in terms of convergence speed but also offer a more refined approach to error estimation, making them a powerful tool for approximation theory.
	

	\end{abstract}
	\section{Introduction}
V. A. Baskakov \cite{Baskakov1957} defined a basic family of linear positive operators in approximation theory, called Baskakov operators, for $f\in C[0, \infty)$ in following form:
\begin{align}
	V_n(f;x)=&\sum_{k=0}^{\infty}p_{n,k}(x)\, f\left(\dfrac{k}{n}\right),\\
\text{where}\,	p_{n,k}(x)=& \binom{n+k-1}{k} \dfrac{x^k}{(1+x)^{n+k}},
\end{align}
 and $C[0, \infty)$ is the class of all continuous functions on $[0,\infty).$  These operators belong to a larger family of positive linear operators that includes well-known operators such as Bernstein \cite{Bernstein1912} and Sz\'asz-Mirakjan \cite{szasz1950general}. This makes them especially helpful in circumstances when precise function representations are unknown or too difficult to work with directly. Baskakov operators employ probabilistic approaches and binomial-type weights to perform function approximation, making them very efficient in certain applications.
Baskakov operators are used in practical applications which require smooth approximations or interpolations. One significant use is numerical analysis, where they contribute to approximate complicated functions for faster calculation, particularly in tasks like as numerical integration or differentiation. These operators are also employed in signal processing to recreate smooth signals from noisy or partial data, hence reducing distortions.

 Although classical Baskakov operators are effective in many approximation applications, but still they have several limitations, such as a sluggish linear rate of convergence and issues approximating functions with greater levels of smoothness. To address these restrictions, researchers created a variety of modifications and generalizations in order to increase their efficiency, widen their applicability, and overcome inherent limits, each adapted to individual purposes. The Baskakov-Durrmeyer operators \cite{sahai1985simu} are well-known generalizations that combine the qualities of both Baskakov and Durrmeyer operators and contain an integral form, making it especially good for approximating smoother functions with better precision. Sahai and Prasad \cite{sahai1985simu} defined these operators in the following way:
 \begin{align}\label{BasDur}
 	D_n(f;x)&=(n-1)\sum_{k=0}^{\infty}p_{n,k}(x)\int_0^{\infty}p_{n,k}(t)\,f(t)\,dt,
 \end{align} and studied their approximation properties. Due to its well known properties, researchers defined other Durrmeyer type operators as one see \cite{finta2008dur, gupta2014different} and the references therein. Another notable generalization is the introduction of q-Baskakov operators, which generalizes the conventional form via q-calculus, giving additional freedom in controlling  convergence behavior and approximation quality. One can see the literature on these generalizations by \cite{agrawal2014appro, aral2010durr, aral2011gene, gupta2011some} Similarly, Stancu-Baskakov operators incorporate an additional parameter, allowing for better control and adaptability, particularly when dealing with functions of varying smoothness.

Also, Boehme and Brucner defined the Baskakov-Kantorovich operators \cite{boeh1964}, which incorporate Kantorovich variant characteristics, resulting in improved handling of non-smooth functions and more precise error estimates. A lot of research is done on these operators \cite{abel2003esti, totik1985satu} and their citations. Extensions of bivariate and multivariate Baskakov operators have been created to approximate functions of several variables, which is beneficial in complicated applications like image processing and computer-aided geometric design.

The order of approximation of the operators must be improved to satisfy the increasing need for higher precision, faster convergence, and enhanced computing efficiency across a wide range of domains. Classical operators, such as the Baskakov or Bernstein operators, often have a slow linear rate of convergence, limiting their ability to approximate complex or smooth functions with acceptable accuracy. Improving the order of these operators not only produces more accurate approximations, but it also speeds up the convergence process, lowering computing costs and increasing time efficiency. This is especially crucial in applications like numerical simulations, solving differential and integral equations, and data interpolation, where little inaccuracies can result in considerable variances. Higher-order operators are also more appropriate for approximating smooth functions with high regularity, as well as handling functions  with irregularities such as discontinuities or sharp gradients. Furthermore, they provide better error estimates, tighter control over approximation accuracy, and higher flexibility, making them more dependable in applications including signal processing, picture reconstruction, and machine learning. Overall, enhancing the order of approximation operators solves the constraints of traditional approaches, allowing for more exact, economical, and adaptive solutions to current mathematical and computing issues.

Recently, in \cite{khosravian2018new}, Khosravian-Arab \emph{et al.} modified the well known Bernstein operators by using a new technique to improve their degree of approximation. Following this, Acu et al. \cite{Acu2019} have applied this approach on the Bernstein-Durrmeyer operators. In another paper \cite{gupta2019modi}, same authors have put it on the Bernstein-Kantorovich operators too. Similarly, Kajla and Acar \cite{kajla2019modified} have modified the $\alpha-$Bernstein summation operators. Following this, Jabbar and Hassan \cite{jabbar2024bet} modified the Baskakov operators with two modifications and studied their properties. This current research has been devoted to improve their convergence rates, leading to modifications and enhancements of the Baskakov-Durrmeyer operators. These improvements enable faster and more accurate approximations, making Baskakov type operators an even more powerful tool in modern approximation theory and its applications.

The present article is organized in the following form:
In Section \ref{first}, we present the first order modification of classical Baskakov-Durrmeyer operators and study their convergence and asymptotic formula. In the next section, we define another modification whose order of convergence is $2.$ We also present its Voronovskaja type result.
\section{First order modification}\label{first}
In this section, we define the modification of Baskakov-Durrmeyer operators \eqref{BasDur} by using its perturbed relation and study its approximation results. The operators are defined as
\begin{align}\label{mod1}
	V_{n,1}(f;x)&=(n-1)\sum_{k=0}^{\infty}p_{n,k}^1(x)\int_0^{\infty}p_{n,k}(t)\,f(t)\,dt,\\
	\text{where}\,\,p_{n,k}^1(x)&=a(x,n)p_{n+1,k}(x)+b(x,n)p_{n+1,k-1}(x)\nonumber\\
	\text{and}\,\, a(x,n)&=a_0(n)+a_1(n)x,\,\, b(x,n)=a_0(n)-a_1(n)(1+x),\nonumber
\end{align}
the sequences $a_0(n)$ and $a_1(n)$ are to be determined in order to satisfy the convergence properties. By choosing $a_0(n)=a_1(n)=1,$ we get the classical Baskakov-Durrmeyer operators.

Now, firstly we find some basic identities for the newly defined operators that are useful for further approximation properties. The moments are a basic tool to check the convergence of the operators by using Korovkin theorem.
\begin{lemma}\label{lemma1}
	 The moments of these operators are listed below:
	\begin{align*}
			V_{n,1}(1;x)=&2a_0(n)-a_1(n);\\
				V_{n,1}(t;x)=&(2a_0(n)-a_1(n))x+\dfrac{1}{n-2}\left[(1+2x)(3a_0(n)-2a_1(n))\right];\\
					V_{n,1}(t^2;x)=&(2a_0(n)-a_1(n))x^2+\dfrac{1}{(n-2)(n-3)}\left[a_0(n)(8+10x-8x^2)+a_1(n)(-6-10x+2x^2)\right.\\
				&	\left.+n\left\{a_0(n)(10x+16x^2)-a_1(n)(6x+10x^2)\right\}\right];\\
						V_{n,1}(t^3;x)=&(2a_0(n)-a_1(n))x^3+\dfrac{1}{(n-2)(n-3)(n-4)}\left[a_0(n)(30+54x+42x^2+60x^3)\right.\\
		     	&+a_1(n)(-24-54x-42x^2-36x^3)+n\left\{a_0(n)(54x+63x^2-30x^3)\right.\\
			    &+\left.\left.a_1(n)(-36x-54x^2+6x^3)\right\}+n^2\left\{a_0(n)(21x^2+30x^3)+a_1(n)(-12x^2-18x^3)\right\}\right];\\
			V_{n,1}(t^4;x)=& (2a_0(n)-a_1(n))x^4+\dfrac{1}{(n-2)(n-3)(n-4)(n-5)}\left[a_0(n)\left(144+336x+382x^2+216x^3-192x^4\right)\right.\\
			&+a_1(n)\left(-120-336x-382x^2-214x^3+72x^4\right)+n\left\{a_0(n)\left(336x+573x^2+396x^3+408x^4\right)\right.\\
			&+\left.a_1(n)\left(-240x-501x^2-361x^3-248x^4\right)\right\}+n^2\left\{a_0(n)\left(191x^2+216x^3-72x^4\right)\right.\\
			&+\left.a_1(n)\left(-119x^2-167x^3+12x^4\right)\right\}+n^3\left\{a_0(n)\left(36x^3+48x^4\right)+a_1(n)\left(-20x^3-28x^4\right)\right\}
			\left.\right].
	\end{align*}
\end{lemma}
\begin{proof}
	In order to prove these identities, we use the following relations and properties of beta function:
	\begin{align*}
		\sum_{k=0}^{\infty}p_{n+1,k}(x)k=&(n+1)x\\
			\sum_{k=0}^{\infty}p_{n+1,k}(x)k^2=& (n+1)(n+2)x^2+(n+1)x\\
			\sum_{k=0}^{\infty}p_{n+1,k}(x)k^3=&(n+1)(n+2)(n+3)x^3+3(n+1)(n+2)x^2+(n+1)x\\
			\sum_{k=0}^{\infty}p_{n+1,k}(x)k^4=&(n+1)(n+2)(n+3)(n+4)x^4+6(n+1)(n+2)(n+3)x^3+7(n+1)(n+2)x^2\\
			&+(n+1)x\\
			\sum_{k=0}^{\infty}p_{n+1,k-1}(x)k=&(n+1)x+1\\
				\sum_{k=0}^{\infty}p_{n+1,k-1}(x)k^2=&(n+1)(n+2)x^2+3(n+1)x+1\\
				\sum_{k=0}^{\infty}p_{n+1,k-1}(x)k^3=&(n+1)(n+2)(n+3)x^3+6(n+1)(n+2)x^2+7(n+1)x+1\\
			\sum_{k=0}^{\infty}p_{n+1,k-1}(x)k^4=&(n+1)(n+2)(n+3)(n+4)x^4+10(n+1)(n+2)(n+3)x^3+24(n+1)(n+2)x^2\\
			&+15(n+1)x+1.
			\end{align*}
	\end{proof}
	In order to achieve $V_{n,1}(1;x)=1,$ we have to choose the unknown sequences $a_0(n)$ and $a_1(n)$ satisfying the following relation:
	\begin{align}\label{seqrel}
		2a_0(n)-a_1(n)=1.
			\end{align}
			Using this relation, we get seven cases for the sequences $a_0(n)$ and $a_1(n)$ listed below:
			\begin{enumerate}[{Case} 1.]
				\item If $a_0(n)=1,$ then $a_1(n)=1.$
				\item If $a_1(n)=0,$ then $a_0(n)=\frac{1}{2}.$
			\item If $a_1(n)>1,$ then $a_0(n)>1.$
			\item If $0< a_1(n)<1,$ then $0< a_0(n)<1.$
			\item If $a_1(n)=-1,$ then $a_0(n)=0.$
			\item If $a_1(n)<-1,$ then $a_0(n)<0.$
			\item If $-1<a_1(n)<0,$ then $0<a_0(n)<\frac{1}{2}.$
			\end{enumerate}	
				It can be easily verified that the operators $V_{n,1}(f;x)$ are positive for the Cases $1-4$ and for $5-6$ the operators are not positive, whereas for Case $7,$ they can have both possibilities.
	\begin{lemma}\label{lemma2}
		The central moments of these operators are given as
		\begin{align*}
				V_{n,1}(t-x;x)=& \dfrac{1}{n-2}[(1+2x)(3a_0(n)-2a_1(n))];\\
					V_{n,1}((t-x)^2;x)=&\dfrac{1}{(n-2)(n-3)}[a_0(n)(-1-8x-8x^2)+a_1(n)2x(1+x)\\
					&+na_0(n)(3+16x+16x^2)+na_1(n)(-2-10x-10x^2)];\\
					V_{n,1}((t-x)^4;x)=&\dfrac{1}{(n-2)(n-3)(n-4)(n-5)}\left[a_0(n)(144+936x+2422x^2+2976x^3+1488x^4)\right.\\
					&\left.+a_1(n)(-120-816x-2182x^2-2734x^3-1368x^4)\right.\\
					&+n\left\{a_0(n)(216x+1005x^2+1584x^3+792x^4)+a_1(n)(-144x-681x^2-1077x^3-540x^4)\right\}\\
					&+\left.n^2\left\{a_0(n)(23x^2+48x^3+24x^4)+a_1(n)(-11x^2-23x^3-12x^4)\right\}\right].
		\end{align*}
	\end{lemma}
	\begin{proof}
		The proof of these identities is easy to prove by using Lemma \ref{lemma1} and linear property of the operators. Therefore, we omit the details.
	\end{proof}	
	\begin{corollary}\label{lemma3} Now, we list some limiting identities for central moments.
		\begin{align*}
			\lim_{n\rightarrow \infty} nV_{n,1}(t-x,x)=&(1+2x)(3l-2m);\\
			\lim_{n\rightarrow \infty} nV_{n,1}((t-x)^2,x)=&l(3+16x+16x^2)-m(2+10x+10x^2); \\
			\lim_{n\rightarrow \infty} nV_{n,1}((t-x)^4,x)=& lx^2(23+48x+24x^2)-mx^2(11+23x+12x^2),
		\end{align*}
		where $l=\displaystyle \lim_{n\rightarrow\infty} a_0(n)$ and $m=\displaystyle \lim_{n\rightarrow \infty}a_1(n).$
	\end{corollary}
	\begin{thm}\label{thm1}
		Let $f\in C_B[0,\infty)$ and the convergence sequences $a_0(n), a_1(n)$ satisfy the \eqref{seqrel}, and the operators are positive, then the operators $	V_{n,1}(f;x)$ \eqref{mod1} converges uniformly to $f(x)$ on a compact subset $A\subset [0,\infty).$
	\end{thm}
	\begin{proof}
	By using Lemma \ref{lemma1}, we can easily check that 
		\begin{align*}
			\lim_{n\rightarrow \infty} 	V_{n,1}(e_0;x)=1;\\
				\lim_{n\rightarrow \infty} 	V_{n,1}(e_1;x)=x;\\
					\lim_{n\rightarrow \infty} 	V_{n,1}(e_2;x)=x^2,
		\end{align*}
		where $e_i=t^i.$ Since the operators are positive, therefore with the help of famous Korovkin theorem \cite{alto2010Kor}, we can prove the uniform convergence of the operators.
	\end{proof}	
	\begin{thm}\label{thm2}
			Let $f\in C_B[0,\infty)$ and the convergence sequences $a_0(n), a_1(n)$ satisfy the \eqref{seqrel} and if the operators are not positive, then the operators satisfy $	\displaystyle \lim_{n\rightarrow \infty} V_{n,1}(f;x)=f(x)$  on a compact subset $A\subset [0,\infty).$
	\end{thm}
	\begin{proof}
		Since the operators are not positive, therefore, we use the extended Korovkin theorem. Let us rewrite the operators $	V_{n,1}(f;x)$ with the combination of positive linear operators in the following way:
		\begin{align}
				V_{n,1}(f;x)=&-A_{n,1}(f;x)-B_{n,1}(f;x),\label{op1}\\
				\text{where} ~~A_{n,1}(f;x)=&(n-1)\sum_{k=0}^{\infty}(-a_1(n)\,x\,p_{n+1,k}(x)+a_1(n)\,p_{n+1,k-1}(x))\int_0^{\infty}p_{n,k}(t)f(t)\,dt,\\
				B_{n,1}(f;x)=&(n-1)\sum_{k=0}^{\infty}(-a_0(n)\,p_{n+1,k}(x)+(a_1(n)\,x-a_0(n))\,p_{n+1,k-1}(x))\int_0^{\infty}p_{n,k}(t)f(t)\,dt.
		\end{align}
		Now, we calculate the moments for $A_{n,1}(f;x)$ and $B_{n,1}(f;x)$ in order to get the uniform convergence.
		\begin{align*}
			A_{n,1}(1;x)=&a_1(n)(1-x);\\
			A_{n,1}(t;x)=&a_1(n)(1-x)\left[x+\dfrac{3x+1}{n-2}\right]+\dfrac{a_1(n)}{n-2};\\
			A_{n,1}(t^2;x)=&a_1(n)(1-x)\left[x^2+\dfrac{8nx^2-4x^2+4(n+1)x}{(n-2)(n-3)}+\dfrac{2a_1(n)n+6a_1(n)}{(n-2)(n-3)}\right],\\
			B_{n,1}(1;x)=&a_1(n)x-2a_0(n);\\
			B_{n,1}(t;x)=&(a_1(n)x-2a_0(n))\left[x+\dfrac{3x}{n-2}\right]+\dfrac{2a_1(n)x-3a_0(n)}{n-2};\\
			B_{n,1}(t^2;x)=&(a_1(n)x-2a_0(n))\left[x^2+\dfrac{8nx^2-4x^2}{(n-2)(n-3)}\right]\\
			&+\dfrac{nx(6a_1(n)-10a_0(n))+9a_1(n)-10a_0(n)x+3a_1(n)-8a_0(n)}{(n-2)(n-3)}.
		\end{align*}
		As the sequences $a_0(n)$ and $a_1(n)$ are convergent. Therefore, by choosing $\displaystyle \lim_{n\rightarrow \infty} a_1(n)=l$ and using Korovkin theorem on $A_{n,1}(f;x)$ and $B_{n,1}(f;x)$, we get
		\begin{align*}
			\lim_{n\rightarrow \infty}A_{n,1}(f;x)=&l(1-x)f(x)\\
			\lim_{n\rightarrow\infty} B_{n,1}(f;x)=&(-l(1-x)-1)f(x).
		\end{align*}
		Now using \eqref{op1}, we get $\displaystyle \lim_{n\rightarrow \infty} V_{n,1}(f;x)=f(x).$\\
		Hence, we get the desired result.
	\end{proof}
	\begin{thm}\label{thm3}
		Let $f\in C_B[0,\infty)$ such that $f^{\prime\prime}(x)$ exists at a certain point $x\in [0,\infty)$ and if the opertaors $V_{n,1}(f;x)$ are positive, then
		\begin{align*}
			\lim_{n\rightarrow\infty}n(V_{n,1}(f;x)-f(x))=&(1+2x)(3l-2m)f^{\prime}(x)+\dfrac{f^{\prime\prime}(x)}{2!}[l(3+16x+16x^2)+m(-2-10x-10x^2)],
		\end{align*}
		where $l$ and $m$ are defined as in Lemma \ref{lemma3}.
	\end{thm}	
	\begin{proof}
		Consider the Taylor's series for function $f$ about $x$, we get
		$$f(t)=f(x)+(t-x)f^{\prime}(x)+\dfrac{(t-x)^2}{2!}f^{\prime\prime}(x)+\xi(t,x)(t-x)^2,$$
		where $\xi\in C[0,\infty)$ such that $\xi(t,x)\rightarrow 0$ as $t\rightarrow x.$
		Apply the operators \eqref{mod1} on Taylor's series and taking limit as $n\rightarrow \infty$
		\begin{align*}
			\lim_{n\rightarrow\infty}n (V_{n,1}(f;x)-f(x))=&f^{\prime}(x)\lim_{n\rightarrow\infty} nV_{n,1}(t-x;x)+\dfrac{f^{\prime\prime}(x)}{2!}\lim_{n\rightarrow\infty} nV_{n,1}((t-x)^2;x)\\
			&+\lim_{n\rightarrow\infty} nV_{n,1}(\xi(t,x)(t-x)^2;x)\\
			=& f^{\prime}(x)(1+2x)(3l-2m)+\dfrac{f^{\prime\prime}(x)}{2!}[l(3+16x+16x^2)+m(-2-10x-10x^2)]\\
			&+\lim_{n\rightarrow\infty}nV_{n,1}(\xi(t,x)(t-x)^2;x).
		\end{align*} 
		Now, we need to prove that $\displaystyle\lim_{n\rightarrow\infty}nV_{n,1}(\xi(t,x)(t-x)^2;x)=0.$ Since the operators are positive, therefore we can use the Cauchy-Schwarz inequality in the following form:
		$$nV_{n,1}(\xi(t,x)(t-x)^2;x)\leq \sqrt{n^2V_{n,1}((t-x)^4;x)}\sqrt{V_{n,1}(\xi^2(t,x);x)}.$$
	 Now, using the continuity of $\xi(t,x)$ and convergence theorem \ref{thm1}, we get
	 $$\lim_{n\rightarrow\infty}V_{n,1}(\xi^2(t,x);x)=\xi^2(x,x)=0.$$
	 Now, using the central moments given in Lemma \ref{lemma2}, we obtained result.
	\end{proof}
	\begin{thm}\label{thm4}
			Let $f\in C_B[0,\infty)$ such that $f^{\prime\prime}(x)$ exists at a certain point $x\in [0,\infty)$ and the sequences $a_0(n)$ and $a_1(n)$ are bounded, then
		\begin{align*}
			\lim_{n\rightarrow\infty}n(V_{n,1}(f;x)-f(x))=&(1+2x)(3l-2m)f^{\prime}(x)+\dfrac{f^{\prime\prime}(x)}{2!}[l(3+16x+16x^2)+m(-2-10x-10x^2)],
		\end{align*}
			where $l$ and $m$ are defined as in Lemma \ref{lemma3}.
	\end{thm}
	\begin{proof}
		The proof of this theorem follows the same approach as in previous Thm \ref{thm3}, therefore we only need to prove that $$\displaystyle\lim_{n\rightarrow\infty}nV_{n,1}(\xi(t,x)(t-x)^2;x)=0.$$
		Therefore, consider \begin{align*}
			V_{n,1}(f;x)=&(n-1)\sum_{k=0}^{\infty}p_{n,k}^1(x)\int_0^{\infty}p_{n,k}(t)\,f(t)\,dt\\
			=&(n-1)\sum_{k=0}^{\infty}(a(x,n)p_{n+1,k}(x)+b(x,n)p_{n+1,k-1}(x))\int_0^{\infty}p_{n,k}(t)\,f(t)\,dt\\
			=&(n-1)\sum_{k=0}^{\infty}a(x,n)p_{n+1,k}(x)\int_0^{\infty}p_{n,k}(t)\,f(t)\,dt\\
			&+ (n-1)\sum_{k=0}^{\infty}b(x,n)p_{n+1,k}(x)\int_0^{\infty}p_{n,k+1}(t)\,f(t)\,dt
		\end{align*}
		 Now,  
		 \begin{align*}
		 	\mid V_{n,1}(\xi(t,x)(t-x)^2;x)\mid \leq & (n-1)\sum_{k=0}^{\infty}\mid a(x,n)\mid p_{n+1,k}(x)\int_0^{\infty}p_{n,k}(t)\,\mid \xi(t,x)\mid (t-x)^2\,dt\\
		 	& +(n-1)\sum_{k=0}^{\infty}\mid b(x,n)\mid p_{n+1,k}(x)\int_0^{\infty}p_{n,k+1}(t)\,\mid \xi(t,x)\mid (t-x)^2\,dt.
		 \end{align*}
		 Since the sequences are bounded, therefore $\mid a(x,n) \mid < C$ and $\mid b(x,n)\mid <C.$ Divide $[0,\infty)$ into two parts $I_1$ and $I_2$ for $\delta >0,$ where
		 $$I_1=(x-\delta,x+\delta)\cap [0,\infty)\quad \text{and}\quad I_2=[0,\infty)\backslash I_1.$$
		 Now, 
		 \begin{align*}
		 		\mid V_{n,1}&(\xi(t,x)(t-x)^2;x)\mid \\
		 		\leq & (n-1)C \sum_{k=0}^{\infty}p_{n+1,k}(x)\left[\int_{I_1}p_{n,k}(t)\mid \xi(t,x)\mid (t-x)^2\,dt+\int_{I_2}p_{n,k}(t)\mid \xi(t,x)\mid (t-x)^2\,dt\right]\\
		 		&+(n-1)C \sum_{k=0}^{\infty}p_{n+1,k}(x)\left[\int_{I_1}p_{n,k+1}(t)\mid \xi(t,x)\mid (t-x)^2\,dt+\int_{I_2}p_{n,k+1}(t)\mid \xi(t,x)\mid (t-x)^2\,dt\right]
		 \end{align*}
		 Now, for $I_1,$ we know that $\mid \xi(t,x)\mid<\epsilon $ and for $I_2,$ it is $\mid\xi(t,x)\mid< \dfrac{M(t-x)^2}{\delta^2}$, where $M:=\sup\{\xi(t,x):t\in I_2\}. $
		 Therefore, we get the relation
		 \begin{align*}
		 \mid V_{n,1}&(\xi(t,x)(t-x)^2;x)\mid \\
		 \leq & (n-1)C \sum_{k=0}^{\infty}p_{n+1,k}(x)\left[\epsilon \int_{I_1}p_{n,k}(t) (t-x)^2\,dt+\dfrac{M}{\delta^2}\int_{I_2}p_{n,k}(t) (t-x)^4\,dt\right]\\
		 &+(n-1)C \sum_{k=0}^{\infty}p_{n+1,k}(x)\left[\epsilon \int_{I_1}p_{n,k+1}(t) (t-x)^2\,dt+\dfrac{M}{\delta^2}\int_{I_2}p_{n,k+1}(t) (t-x)^4\,dt\right]\\
		 \leq & \dfrac{4C\epsilon}{(n-2)(n-3)}\left[nx(x+1)+7x^2+7x+2\right]+\dfrac{24MC}{{\delta}^2(n-2)(n-3)(n-4)(n-5)}\left[n^2x^2(x+1)^2\right.\\
		 &\left.+nx(9+42x+66x^2+33x^3)+6+39x+101x^2+124x^3+62x^4\right]
		 \end{align*}
		 Since the sequences $a_0(n)$ and $a_1(n)$ are bounded, therefore we get $\displaystyle\lim_{n\rightarrow\infty}nV_{n,1}(\xi(t,x)(t-x)^2;x)=0.$
		 Hence, the result.
	\end{proof}	
		\section{Second order modification}
		In the previous section, we introduce the first order modification of Baskakov-Duurmeyer operators, which has the linear rate of convergence. Now, inspiring from the perturbation of Baskakov operators, we present an another modification whose rate of convergence is better than first order modification.
		\begin{align*}
			V_{n,2}(f;x)=& (n-1)\sum_{k=0}^{\infty}p_{n,k}^2(x)\int_0^{\infty}p_{n,k}(t)\, f(t)\, dt,
		\end{align*} 
		where $p_{n,k}^2(x)=a(x,n)p_{n+2,k}(x)+d(x,n)p_{n+2,k-1}(x)+a^{\prime}(x,n)p_{n+2,k-2}(x),$\\
		$a(x,n)=a(n)+b(n)x+c(n)x^2, \, d(x,n)=d(n)x(1+x),\, a^{\prime}(x,n)=a(n)-b(n)(1+x)+c(n)(1+x)^2,$\\
		where $a(n), b(n), c(n)$ and $d(n)$ are the unknown sequences to be determined to serve our purposes. Now, by letting $a(n)=1=c(n), b(n)=2,$ and $d(n)=-2,$ we get the classical Baskakov-Durrmeyer operators.
		Firstly, we calculate the moments of these operators.
		\begin{align*}
			V_{n,2}(1;x)=&\sum_{k=0}^{\infty}p_{n,k}^2(x)= a(x,n)+d(x,n)+a^{\prime}(x,n)\\
				=&(2a(n)-b(n)+c(n))+x(d(n)+2c(n))+x^2(d(n)+2c(n)).
		\end{align*}
		In order to achieve $V_{n,2}(1;x)=1,$ we need to choose the unknown sequences that satisfies the relation 
		\begin{align}\label{relation1}
			2a(n)-b(n)+c(n)=1\,\, \text{and}\,\, d(n)=-2c(n).
				\end{align}
		Now, \begin{align}\label{eq.2.1}
			V_{n,2}(t;x)=&\sum_{k=0}^{\infty} p_{n,k}^2(x)\dfrac{k+1}{n-2}.
		\end{align}
		By using \eqref{relation1} and the following identities:
		\begin{align*}
				\sum_{k=0}^{\infty} p_{n+2,k}(x)&=1;\\
					\sum_{k=0}^{\infty} p_{n+2,k-1}(x)&=1;\\
						\sum_{k=0}^{\infty} p_{n+2,k-2}(x)&=1;\\
			\sum_{k=0}^{\infty} p_{n+2,k}(x)k&=(n+2)x;\\
				\sum_{k=0}^{\infty} p_{n+2,k-1}(x)k&=(n+2)x+1;\\
					\sum_{k=0}^{\infty} p_{n+2,k-2}(x)k&=(n+2)x+2,
		\end{align*}
		the \eqref{eq.2.1} becomes
		\begin{align*}
			V_{n,2}(t;x)=x+\dfrac{1}{n-2}[(3-2a(n))(1+2x)].
		\end{align*}
		In order to get $V_{n,2}(t;x)=x,$ we need to choose $a(n)=\frac{3}{2}.$
		Now, for second order moment, we can see that
		$$V_{n,2}(t^2;x)=\sum_{k=0}^{\infty}p_{n,k}^2(x)\dfrac{(k+1)(k+2)}{(n-2)(n-3)}.$$
		Again using the following identities with \eqref{eq.2.1}:
		\begin{align*}
				\sum_{k=0}^{\infty} p_{n+2,k}(x)k^2&=(n+3)(n+2)x^2+(n+2)x;\\
					\sum_{k=0}^{\infty} p_{n+2,k-1}(x)k^2&=(n+3)(n+2)x^2+3(n+2)x+1;\\
						\sum_{k=0}^{\infty} p_{n+2,k-2}(x)k^2&=(n+3)(n+2)x^2+5(n+2)x+4,
		\end{align*}
		we get
		$$V_{n,2}(t^2;x)=x^2+\dfrac{1}{(n-2)(n-3)}[-3-16x-16x^2+2x(n+c(n))+2x^2(n+c(n))].$$
		In order to get second order modification, we need to choose $c(n)=-n.$ Therefore, we choose all the unknown sequences $a(n)=\frac{3}{2}$, $b(n)=2-n,$ $c(n)=-n,$ and $d(n)=2n.$\\
		Hence, the operators becomes
		\begin{align}
			\hat{V}_{n,2}(f;x)=&(n-1)\left[\left(\frac{3}{2}+2x-nx(1+x)\right)\sum_{k=0}^{\infty}p_{n+2,k}(x)\int_0^{\infty}p_{n,k}(t)f(t)\,dt\right.\\
	&+2nx(1+x)\sum_{k=0}^{\infty}p_{n+2,k-1}(x)\int_0^{\infty}p_{n,k}(t)f(t)\,dt\\
	&+ 	\left.\left(-\frac{1}{2}-2x-nx(1+x)\right)\sum_{k=0}^{\infty}p_{n+2, k-2}(x)\int_0^{\infty}p_{n,k}(t)f(t)\,dt\right].
		\end{align} 
		\begin{lemma} The moments of the operators $	\hat{V}_{n,2}(f;x)$ are as follow:
			\begin{align*}
		\hat{V}_{n,2}(1;x)=&1;\\
		\hat{V}_{n,2}(t;x)=&x;\\
		\hat{V}_{n,2}(t^2;x)=&x^2+\dfrac{1}{(n-2)(n-3)}[-3-16x-16x^2];\\
		\hat{V}_{n,2}(t^3;x)=& \dfrac{1}{(n-2)(n-3)(n-4)}\left[x^3(n^3-9n^2-46n-48)+x^2(-84n-132)\right.\\
		&\left.+x(-21n-114)-21\right];\\
		\hat{V}_{n,2}(t^4;x)=&\dfrac{1}{(n-2)(n-3)(n-4)(n-5)}\left[x^4(n^4-14n^3-133n^2-334n-264)+x^3(-264n^2-1032n-1008)\right.\\
	&	\left.+x^2(-78n^2-1014n-1428)+x(-240n-864)-144\right];\\
	\hat{V}_{n,2}(t^5;x)=&\dfrac{1}{(n-2)(n-3)(n-4)(n-5)(n-6)}\left[x^5(n^5-20n^4-305n^3-1360n^2-2516n-1680)\right.\\
	&x^4(-640n^3-4560n^2-10640n-8160)+x^3(-210n^3-4890n^2-16860n-15840)+x^2(-1350n^2-11550n-15300)\\
	&+\left.x(-2400n-7200)-1080\right];\\
	\hat{V}_{n,2}(t^6;x)=&\dfrac{1}{(n-2)(n-3)(n-4)(n-5)(n-6)(n-7)}\left[x^6(n^6-27n^5-605n^4-4185n^3-13436n^2-20628n-12240)\right.\\
	&+x^5(-1320n^4-14880n^3-61320n^2-109680n-72000)+x^4(-465n^4-16950n^3-105375n^2-235050n-176760)\\
	&+x^3(-5160n^3-78840n^2-252960n-231840)+x^2(-18900n^2-134100n-171000)\\
	&\left.+x(-24480n-66240)-9000\right].
		\end{align*}
		\end{lemma}
		\begin{lemma} The central moments of the operators $	\hat{V}_{n,2}(f;x)$ are as follow:
			\begin{align*}
			\hat{V}_{n,2}(t-x;x)=&0;\\
			\hat{V}_{n,2}((t-x)^2;x)=&\dfrac{1}{(n-2)(n-3)}[-3-16x-16x^2];\\
			\hat{V}_{n,2}((t-x)^3;x)=&\dfrac{1}{(n-2)(n-3)(n-4)}\left[n(-12x-36x^2-24x^3)-21-150x-324x^2-216x^3\right];\\
			\hat{V}_{n,2}((t-x)^4;x)=&\dfrac{1}{(n-2)(n-3)(n-4)(n-5)}\left[n^2(-12x^2-24x^3-12x^4)\right.\\
			&\left.+n(-156x-816x^2-1320x^3-660x^4)-144-1284x-4068x^2-5568x^3-2784x^4\right]\\
			\hat{V}_{n,2}((t-x)^5;x)=&\dfrac{1}{(n-2)(n-3)(n-4)(n-5)(n-6)}\left[n^2(-360x^2-1440x^3-1800x^4-720x^5)\right.\\
			&+n(-1680x-12120x^2-31680x^3-35400x^4-14160x^5)\\
			&\left.-1080-11520x-47520x^2-96480x^3-97200x^4-38880x^5\right]\\
			\hat{V}_{n,2}((t-x)^6;x)=&\dfrac{1}{(n-2)(n-3)(n-4)(n-)(n-6)(n-7)}\left[n^3(-240x^3-720x^4-720x^5-240x^6)\right.\\
			&n^2(-6660x^2-39960x^3-86580x^4-79920x^5-26640x^6)\\
			&+n(-18000x-163620x^2-584040x^3-1024020x^4-878400x^5-292800x^6)\\
			&\left.-9000-111600x-564120x^2-1506960x^3-2258280x^4-1805760x^5-601920x^6\right].
			\end{align*}
		\end{lemma}
		\begin{thm}
				Let $f\in C_B[0,\infty)$ such that $f^{\prime\prime}(x)$ and $f^{\prime\prime\prime}(x)$ exist at a certain point $x\in [0,\infty).$ Then, we get
			\begin{align*}
				\lim_{n\rightarrow\infty}(	\hat{V}_{n,2}(f;x)-f(x))=&O\left(\frac{1}{n^2}\right).
			\end{align*}
		\end{thm}
	\bibliographystyle{abbrv}
	\bibliography{bibfile2}
%
%
%

\end{document}